\documentclass[11pt, letterpaper]{amsart}

\usepackage{graphicx}
\usepackage{booktabs} 
\usepackage[figuresright]{rotating}

\newtheorem{theorem}{\noindent{\bf Theorem}}[section]
\newtheorem{definition}[]{\noindent{\bf Definition}}[section]
\newtheorem{proposition}[]{\noindent{\bf Proposition}}[section]

\newtheorem{lemma}[]{\noindent{\bf Lemma}}[section]
\newtheorem{corollary}[]{\noindent{\bf Corollary}}[section]
\newtheorem{example}[]{\noindent{\bf Example}}[section]
\newtheorem{remark}[]{\noindent{\bf Remark}}[section]

\begin{document}

\title[Existence of Branched covers $S^{2}\rightarrow S^{2}$ with prescribed branching data]{Existence of Branched covers $S^{2}\rightarrow S^{2}$ with prescribed branching data}
\author {Yingjie Meng,~~Zhiqiang Wei,~~ Chuankai Zhou}


\begin{abstract}
Building on techniques from complex analysis and topology, we establish a remarkable property of branched covers and formulate a complete criterion for the existence of specific types of branched covers between 2-spheres. Our results extend and unify previous work by Jiang (2004), Pervova-Petronio (2006), Zhu (2019), and Wei-Wu-Xu (2024). As applications of our criterion, we present several new families of exceptional branching data.
\vspace*{2mm}

\noindent{\bf Key words}\hskip3mm  Branched cover, branching data, the Hurwitz existence problem.
\vspace*{2mm}\\
\noindent{\bf 2020 MR Subject Classification:}\hskip3mm 57M12.
\thispagestyle{empty}

\end{abstract}
\maketitle

\section{Introduction}
\setcounter{equation}{0}
The study of branched covers between compact, connected, orientable surfaces with prescribed branching data represents a fundamental problem in complex analysis and low-dimensional topology, with its origins tracing back to Hurwitz's foundational work \cite{Hur91}. This classical problem, known as the Hurwitz existence problem, continues to play a central role in geometric function theory and surface topology.

In his seminal 1891 paper, Hurwitz \cite{Hur91} established a remarkable correspondence between the existence of branched covers and factorizations of permutations in symmetric groups. While theoretically elegant, this algebraic characterization presents significant computational challenges in practical applications. A major breakthrough was achieved by Edmonds, Kulkarni, and Stong \cite{EKS84}, who provided a complete solution for the case when the target surface $X$ satisfies $\chi(X)\leq 0$. However, the special case of $X=S^2$ (i.e. the Riemann sphere) exhibits unexpected complexity-for example, the branching data $\mathcal{D}=\{[3,1],[2,2],[2,2]\}$ proves to be unrealizable. Their work led to the formulation of the influential Prime Degree Conjecture, which has received substantial supporting evidence through subsequent research by Pakovich \cite{Pak09} and Pascali-Petronio \cite{PP09,PP12}, yet remains unresolved to this day.

The literature reveals several distinct approaches to this problem. Bara\'{n}ski \cite{Bar01} pioneered the use of graph-theoretic methods on surfaces to characterize realizable branching data for $S^2$. Jiang \cite{JY04} employed differential topological techniques to establish necessary and sufficient conditions for certain classes of branching data. Zheng's computational work \cite{Zhe06} systematically classified exceptional data for low-degree covers ($d\leq22$) between $S^2$ and $S^2$ with three branch points. Pervova and Petronio \cite{PP06} combined techniques from dessins d'enfants, decomposability theory, and surface graphs to produce novel examples of both realizable and exceptional data. Most recently, Wei, Wu, and Xu \cite{WWX24}, using the decomposability approach, found new properties of  branched covers and discovered additional exceptional cases with three branch points between $S^2$ and $S^2$. Pakovich's latest contribution \cite{FP24} leverages sophisticated methods from holomorphic dynamics and fiber product constructions to generate extensive families of new examples. For comprehensive surveys of related results, we direct readers to \cite{Bo82,CH22,EKS84,AD84,PP08,P20,SX20,Th65,Zhu19} and references therein.

Despite these significant advances, a complete classification of realizable branching data remains an open challenge. In this work, we focus on branched covers $f\colon S^2\to S^2$ with at least three branch points, combining Sto\"{i}low's theorem with differential topological methods to establish a crucial structural property (Proposition \ref{main-p1}) and some existence results. Our main results are summarized as follows:

\begin{proposition}\label{main-p1}
Let $s\geq2$, $t\geq1$, and $d'\geq1$ be integers, and consider $n\geq3$ nontrivial partitions of $sd'$ given by:
\begin{align*}
A_1 &= [sa_{11},sa_{12},\ldots,sa_{1r_1}], \\
A_2 &= [sa_{21},sa_{22},\ldots,sa_{2r_2}], \\
A_3 &= [ta_{31},ta_{32},\ldots,ta_{3r_3}], \\
A_k &= [a_{n1},a_{n2},\ldots,a_{nr_n}]~~(4\leq k\leq n),
\end{align*}
where all $a_{kl}\geq1$ are integers. If there exists a branched cover $f:S^{2}\rightarrow S^{2}$ with branching data $\mathcal{D}=\{A_1,A_2,\ldots,A_n\}$, then the following holds:
\begin{enumerate}
\item For $s\geq4$, necessarily $t=1$.
\item For $s=3$, $t\in\{1,2\}$. Moreover, when $t=2$, we have $4\mid d'$.
\item For $s=2$, $t\geq1$. Furthermore, if $t\geq2$, then $t$ divides $d'$.
\end{enumerate}
\end{proposition}

\begin{theorem}\label{main-th1}
Let $s\geq2$ and $d'\geq1$ be integers. Consider $n\geq3$ nontrivial partitions of $sd'$:
\begin{align*}
A_1 &= [sa_{11},\ldots,sa_{1r_{1}}], \\
A_2 &= [sa_{21},\ldots,sa_{2r_{2}}], \\
A_k &= [a_{k1},\ldots,a_{kr_k}] \quad (3\leq k\leq n),
\end{align*}
satisfying the Riemann-Hurwitz condition $\sum\limits_{k=1}^{n}r_k = s(n-2)d'+2$. Then the following are equivalent:
\begin{enumerate}
\item There exists a degree $sd'$ branched cover $f:S^2\to S^2$ with branching data $\mathcal{D}=\{A_1,\ldots,A_n\}$;
\item There exists a degree $d'$ branched cover $g:S^2\to S^2$ whose branching data contains $[a_{11},\ldots,a_{1r_1}]$ and $[a_{21},\ldots,a_{2r_2}]$, with the remaining data obtained by decomposing each $A_k$ ($k\geq3$) into $s$ partitions of $d'$.
\end{enumerate}
\end{theorem}

\begin{theorem}\label{main-th2}
Let $d'\geq2$ and $t\geq2$ be integers with $t\mid d'$. Consider $n\geq3$ nontrivial partitions of $2d'$:
\begin{align*}
A_1 &= [2a_{11},\ldots,2a_{1r_1}], \\
A_2 &= [2a_{21},\ldots,2a_{2r_2}], \\
A_3 &= [ta_{31},\ldots,ta_{3r_3}], \\
A_k &= [a_{k1},\ldots,a_{kr_k}] \quad (4\leq k\leq n),
\end{align*}
satisfying $\sum\limits_{k=1}^n r_k = 2(n-2)d'+2$. Then the following are equivalent:
\begin{enumerate}
\item There exists a degree $2d'$ branched cover $f:S^2\to S^2$ with branching data $\mathcal{D}=\{A_1,\ldots,A_n\}$;
\item There exists a degree $d'/t$ branched cover $g:S^2\to S^2$ whose branching data contains two partitions obtained from $[a_{31},\ldots, a_{3r_3}]$, along with data from:
\begin{itemize}
\item Decomposing each $[a_{1i}]$ and $[a_{2i}]$ into $t$ partitions of $d'/t$;
\item Decomposing each $A_k$ ($k\geq4$) into $2t$ partitions of $d'/t$.
\end{itemize}
\end{enumerate}
\end{theorem}

\begin{theorem}\label{main-th3}
Let $d'\geq4$ with $4\mid d'$. Consider $n\geq3$ nontrivial partitions of $3d'$:
\begin{align*}
A_1 &= [3a_{11},\ldots,3a_{1r_1}], \\
A_2 &= [3a_{21},\ldots,3a_{2r_2}], \\
A_3 &= [2a_{31},\ldots,2a_{3r_3}], \\
A_k &= [a_{k1},\ldots,a_{kr_k}] \quad (4\leq k\leq n),
\end{align*}
satisfying $\sum\limits_{k=1}^n r_k = 3(n-2)d'+2$. Then the following are equivalent:
\begin{enumerate}
\item There exists a degree $3d'$ branched cover $f:S^2\to S^2$ with branching data $\mathcal{D}=\{A_1,\ldots,A_n\}$;
\item There exists a degree $d'/4$ branched cover $g:S^2\to S^2$ whose branching data contains six partitions obtained from $$[a_{31},\ldots,a_{3r_3}],$$
    along with data from:
\begin{itemize}
\item Decomposing each $[a_{1i}]$ and $[a_{2i}]$ into 4 partitions of $d'/4$;
\item Decomposing each $A_k$ ($k\geq4$) into 12 partitions of $d'/4$.
\end{itemize}
\end{enumerate}
\end{theorem}

\begin{remark}
The results above synthesize and extend several important precedents in the literature:
\begin{itemize}
\item The $n=3$ case in theorem \ref{main-th1} specializes to the classification established by Wei, Wu, and Xu \cite{WWX24};
\item When $s=2$ with $n\geq3$ in theorem \ref{main-th1}, we recover Pervova and Petronio's earlier result \cite{PP06};
\item Theorem \ref{main-th2} provides an equivalent formulation of Jiang's criterion \cite{JY04} through a novel decomposition approach.
\end{itemize}
\end{remark}

By these results, one would have immediate corollaries for realizable data between $2$-spheres.

\begin{corollary}\label{Co-1}
Under the assumptions of Theorem \ref{main-th1}, we have:
\begin{enumerate}
\item For all $i,j$, the parts satisfy $\alpha_{ij} \leq d'$;
\item Each partition $A_i$($i\geq3$) has length $r_i>s$.

\end{enumerate}
\end{corollary}

\begin{corollary}\label{Co-2}
Under the assumptions of Theorem \ref{main-th2}, we have:
\begin{enumerate}
\item All parts satisfy $\alpha_{ij} \leq d'/t$;
\item If $n\geq4$, Each partition $A_i$($i\geq4$) has length $r_i>2t$.
\end{enumerate}
\end{corollary}

\begin{corollary}\label{Co-3}
Under the assumptions of Theorem \ref{main-th3}, we have:
\begin{enumerate}
\item All parts satisfy $\alpha_{ij} \leq d'/4$;
\item If $n\geq4$, Each partition $A_i$($i\geq4$) has length $r_i>12$.
\end{enumerate}
\end{corollary}

A natural question arises regarding the flexibility of branching data: Given a realizable branched cover $f \colon S^2 \to S^2$ with branching datum
\[
\mathcal{D} = \{[\alpha_1, \ldots, \alpha_A], [\beta_1, \ldots, \beta_B], [\gamma_1, \ldots, \gamma_C]\},
\]
can we arbitrarily modify one partition while keeping the other two fixed? Specifically, for any modified candidate
\[
\widetilde{\mathcal{D}} = \{[\alpha_1, \ldots, \alpha_A], [\beta_1, \ldots, \beta_B], [\widetilde{\gamma}_1, \ldots, \widetilde{\gamma}_C]\},
\]
does there exist a corresponding branched cover $g \colon S^2 \to S^2$ with datum $\widetilde{\mathcal{D}}$?

The answer is negative in general. For example, Zheng's work \cite{Zhe06} demonstrates that:
\begin{itemize}
\item The datum $\{[5,3], [2,2,2,2], [3,2,2,1]\}$ is realizable;
\item However, the modified datum $\{[5,3], [2,2,2,2], [3,3,1,1]\}$ cannot be realized by any branched cover.
\end{itemize}
This shows that the realizability of branching data exhibits non-trivial constraints beyond just the Riemann-Hurwitz condition.

At last, the following theorem  follows directly from Theorem \ref{main-th1} and a theorem of Song and Xu \cite{SX20}.

\begin{theorem}
Let $k\geq3$, $x\geq1$, and $y\geq1$ be integers satisfying $k\geq\max\{x,y\}$. There exists a branched cover $f\colon\overline{\mathbb{C}}\to\overline{\mathbb{C}}$ between Riemann spheres with branching data
\[
\Big\{[a_1,\ldots,a_{x+y}],\, [\underbrace{2,\ldots,2}_{k-y},2y],\, [\underbrace{2,\ldots,2}_{k-x},2x]\Big\}
\]
if and only if:
\begin{enumerate}
\item The partition $[a_1,\ldots,a_{x+y}]$ can be decomposed into two partitions of $k$;
\item The integer $k$ satisfies $\frac{k}{\gcd(a_1,\ldots,a_{x+y})} \geq \max\{x,y\}$.
\end{enumerate}
\end{theorem}

The paper is organized as follows: Section 2 reviews fundamental concepts and preliminary results. Section 3 presents detailed proofs of our results. Section 4 provides new examples of exceptional branching data, illustrating the application of our theoretical results.

\section{Preliminaries}
\label{sec:preliminaries}
This section establishes the foundational framework for branched covers between compact, connected, orientable surfaces (i.e. Riemann surfaces) and presents key results essential for our subsequent analysis. Our notation and terminology primarily follow \cite{PP09}.

\subsection{Branched Covers}
Let $X$ and $Y$ be compact, connected, orientable surfaces. A \emph{branched cover} $f:Y\rightarrow X$ is a surjective map that is locally modeled on power maps $\mathbb{C}\ni z\mapsto z^k\in\mathbb{C}$ for positive integers $k$. The integer $k$ is called the \emph{local degree} at the preimage point corresponding to $0\in\mathbb{C}$. When $k>1$, the image point in $X$ is called a \emph{branch point}.

Key properties of branched covers include:
\begin{itemize}
\item Branch points are isolated, hence for any branched cover $f$ there exist finitely many (say $n$) branch points
\item Removing all branch points in $X$ and their preimages in $Y$ yields a genuine covering map of degree $d$
\item The local degrees at preimages of each branch point form a nontrivial partition $A_i$ of $d$ (a multiset of positive integers summing to $d$ with at least one element $>1$)
\end{itemize}

We denote by $l(A_i)$ the length of partition $A_i$, and call $\mathcal{D}(f)=\{A_1,\ldots,A_n\}$ the \emph{branching data} of $f$.

\subsection{The Hurwitz Existence Problem}
For a branched cover $f:Y\rightarrow X$ between compact, connected, orientable surfaces with Euler characteristics $\chi(Y)$ and $\chi(X)$ respectively, the following fundamental relation holds:

\begin{theorem}[Riemann-Hurwitz Formula]
\label{thm:rhf}
For any degree $d$ branched cover $f:Y\rightarrow X$ with branching data $\mathcal{D}=\{A_1,\ldots,A_n\}$, we have:
\begin{equation}
\label{eq:rhf}
d\chi(X) - \chi(Y) = nd - \sum_{i=1}^n l(A_i)
\end{equation}
\end{theorem}

 The \emph{Hurwitz existence problem} asks whether, given two compact, connected and orientable surfaces $Y$ and $X$ with Euler characteristics $\chi(Y)$ and
 $\chi(X)$ respectively, an integer $d\geq2$, and $n$ nontrivial partitions $A_{1},\ldots,A_{n}$ of
 $d$ which satisfy the Riemann-Hurwitz formula \eqref{eq:rhf}, there exists a branched cover
 $f:Y\rightarrow X$ such that $\mathcal{D}=\{A_1,\ldots,A_n\}$ is the branching data of $f$. Thus we have the following notation.

\begin{definition}
Given compact, connected, orientable surfaces $X,Y$ with Euler characteristics $\chi(X),\chi(Y)$, a positive integer $d\geq 2$, and $n$ nontrivial partitions $A_1,\ldots,A_n$ of $d$ satisfying the Riemann-Hurwitz formula \eqref{eq:rhf}, we call $\mathcal{D}=\{A_1,\ldots,A_n\}$ a candidate branching data. If there exists a branched cover $f:Y\rightarrow X$ realizing $\mathcal{D}$, we say $\mathcal{D}$ is realizable; otherwise, it is exceptional.
\end{definition}

\subsection{Riemann Existence Theorem}
The classical work of Hurwitz \cite{Hur91} established a profound connection between branched coverings and group-theoretic data.

\begin{theorem}[Riemann Existence Theorem, \cite{SK11,Hur91}]
\label{thm:ret}
Let $X$ be a connected Riemann surface with discrete subset $\Delta\subset X$. For any:
\begin{itemize}
\item Positive integer $d\geq 1$
\item Transitive permutation representation $\rho:\pi_1(X\setminus\Delta)\rightarrow S_d$
\end{itemize}
there exists a unique (up to equivalence) connected Riemann surface $Y$ and proper holomorphic map $f:Y\rightarrow X$ whose monodromy homomorphism is $\rho$.
\end{theorem}

This theorem translates the topological problem of existence of branched covers into an algebraic problem about permutation representations. However, this algebraic characterization presents significant computational challenges in practical applications.

\subsection{Sto\"{i}low's Theorem}
The classification of topological branched covers is facilitated by Sto\"{i}low's Theorem, which need the following notation.

\begin{definition}
A continuous map $f:Y\rightarrow X$ between surfaces is:
\begin{itemize}
\item \emph{Light} if $f^{-1}(x)$ is totally disconnected for all $x\in X$
\item \emph{Discrete} if $f^{-1}(x)$ is discrete for all $x\in X$
\item \emph{Open} if it maps open sets to open sets
\end{itemize}
The \emph{branch set} $B_f\subset Y$ consists of points where $f$ fails to be a local homeomorphism.
\end{definition}

\begin{theorem}[Sto\"{i}low's Theorem, \cite{RP19,St38,Wh42,Wh64}]\label{St1928}
\label{thm:stoilow}
For any continuous, open, light map $f:Y\rightarrow X$ between Riemann surfaces:
\begin{enumerate}
\item $f$ is discrete with discrete branch set $B_f$
\item There exists a Riemann surface $\widetilde{Y}$ and homeomorphism $h:Y\rightarrow\widetilde{Y}$ such that $f\circ h^{-1}:\widetilde{Y}\rightarrow X$ is holomorphic
\end{enumerate}
\end{theorem}

This establishes that topological branched covers between orientable surfaces are essentially equivalent to holomorphic ones, being locally modeled by power maps $z\mapsto z^k$.

\section{Proofs of Main Results}
\label{sec:proofs}

\subsection{Preliminary Lemmas}
We begin by establishing three fundamental lemmas that form the technical foundation for our proofs. First, we recall a
 well known result in the field of topology.

\begin{theorem}\label{ZS96}
Let $X$ be a connected smooth manifold with dimension $\geq2 $. Let $p_1,\ldots,p_n$  be $n$ distinct points in $X$. Then for any $n$ distinct points $q_1,\ldots,q_n$ in $X$ there exists a smooth homeomorphism $h:X\rightarrow X$ that is homotopic to the identity map and satisfies $h(p_i)=q_i$, for $i=1,\ldots,n$.
\end{theorem}

From \textbf{Theorems \ref{ZS96}} and \textbf{\ref{St1928}}, we obtain the following lemmas:
\begin{lemma}
Let $f:Y\rightarrow X$ be a branched cover between two compact connected Riemann surfaces $Y$ and $X$ with branching data $\mathcal{D}=\{A_1,\ldots,A_n\}$. Suppose the branching points of $f$ in $X$ are $p_1,\cdots,p_n$. Then for any $n$ distinct points $q_1,\ldots,q_n$ in
$X$ there exists a branched cover $g:Y\rightarrow X$ such that its branching data is $\mathcal{D}=\{A_1,\ldots,A_n\}$ and its branching points are $q_1,\cdots,q_n$.
\end{lemma}

\begin{proof}
By \textbf{Theorem \ref{ZS96}}, there exists a smooth homeomorphism $h:X\rightarrow X$  such that $h(p_i)=q_i$ for $i=1,\ldots,n$. Obviously, $h\circ f:Y\rightarrow X$ is the desired map.
\end{proof}

\begin{lemma}
Let $f:Y\rightarrow X$ be a branched cover between two compact connected Riemann surfaces $Y$ and $X$. Then there exists a Riemann surface $\widetilde{Y}$ and a homeomorphism $h:Y\rightarrow\widetilde{Y}$ such that the composition $f\circ h^{-1}:\widetilde{Y}\rightarrow X$ is a holomorphic map.
\end{lemma}

\begin{proof}
This result is directly derived from \textbf{Theorem \ref{St1928}}.
\end{proof}

Furthermore, when $Y=X=S^{2}$, given that the complex structure on $S^{2}$ is unique, we can deduce the following lemma.

\begin{lemma}\label{main-le1}
Let $f:S^{2}\rightarrow S^{2}$ be a branched cover between the Riemann spheres $S^{2}$ with $n$ branch points. Then for any distinct $n$ points $p_{1},\ldots,p_{n}$ in $S^{2}$ there exists a holomorphic branched cover $g:S^{2}\rightarrow S^{2}$  such that its branching data is identical to that of $f$, and $p_{1},\ldots,p_{n}$ are all branch points of $g$.
\end{lemma}

\subsection{Proof of Proposition \ref{main-p1}}

\begin{proof}
Given the existence of a degree $sd'$ branched cover $f:S^{2}\rightarrow S^{2}$ with branching data $\mathcal{D}=\{A_1,A_2,A_3,\ldots,A_n\}$, the Riemann-Hurwitz formula yields:
\begin{equation}
\sum_{k=1}^{n}r_{k} = s(n-2)d' + 2.
\end{equation}

Considering the following inequalities on the ramification indices:
\begin{equation}
r_1 \leq d',\quad r_2 \leq d',\quad r_3 \leq \frac{sd'}{t},\quad r_k \leq sd'-1 \text{ for } k \geq 4,
\end{equation}
we derive the upper bound:
\begin{equation}
s(n-2)d' + 2 = \sum_{k=1}^{n}r_{k} \leq 2d' + \frac{sd'}{t} + (n-3)(sd'-1).
\end{equation}
This inequality simplifies to:
\begin{equation}
n-1 \leq \left(2 + \frac{s}{t} - s\right)d'.
\end{equation}

Define the parameter function:
\begin{equation}
\tau(s,t) = 2 + \frac{s}{t} - s \quad \text{for} \quad s \geq 2, t \geq 1.
\end{equation}

From the inequality $0 < n-1 \leq \tau(s,t)d'$, we conclude that $\tau(s,t) > 0$. However, when $s \geq 3$ and $t \geq 3$, we have $\tau(s,t) \leq \tau(3,3) = 0$, contradicting the positivity requirement. This leads to the following cases:
\begin{enumerate}
    \item For $s \geq 4$, necessarily $t = 1$ (since $t=2$ would give $\tau(s,2) = 2-s/2 \leq 0$ when $s \geq 4$).
    \item For $s = 3$, $t$ must be either $1$ or $2$.
    \item For $s = 2$, any $t \geq 1$ is admissible since $\tau(2,t) = 2/t > 0$.
\end{enumerate}

For the specific case when $s=3$ and $t=2$, the partitions take the form:
\begin{equation*}
A_1 = [3a_{11},\ldots,3a_{1r_{1}}], \quad A_2 = [3a_{21},\ldots,3a_{2r_{2}}], \quad A_3 = [2a_{31},\ldots,2a_{3r_{3}}].
\end{equation*}

Identifying $S^{2}$ with the Riemann sphere $\mathbb{C}\cup\{\infty\}$, Lemma \ref{main-le1} allows us to assume $f$ is a rational map with branching points at $0$ and $\infty$ corresponding to $A_1$ and $A_2$ respectively. The divisibility condition implies the existence of a degree $d'$ rational map $h$ with branching data:
\begin{itemize}
    \item Partitions $[a_{11},\ldots,a_{1r_{1}}]$ and $[a_{21},\ldots,a_{2r_{2}}]$
    \item Data obtained by splitting each $A_k$ ($k \geq 3$) into three partitions of $d'$
\end{itemize}

Let $A_3$ split into three partitions:
\begin{equation*}
A_{31} = [2a^{1}_{31},\ldots,2a^{1}_{3r_{1}}], \quad A_{32} = [2a^{2}_{31},\ldots,2a^{2}_{3r_{2}}], \quad A_{33} = [2a^{3}_{31},\ldots,2a^{3}_{3r_{3}}].
\end{equation*}

The common factor of 2 in $A_{31}$ and $A_{32}$ implies the existence of a degree $\frac{d'}{2}$ rational map $g$ with branching data from:
\begin{itemize}
    \item Partitions $[a^{1}_{31},\ldots,a^{1}_{3r_{1}}]$ and $[a^{2}_{31},\ldots,a^{2}_{3r_{2}}]$
    \item Splitting $A_k$ ($k \geq 4$) into six partitions of $\frac{d'}{2}$
    \item Splitting $A_{33}$ into two partitions of $\frac{d'}{2}$, denoted by $A_{331}$ and $A_{332}$
\end{itemize}

The remaining partitions $A_{331}$ and $A_{332}$ again share a common factor of 2, implying the existence of a degree $\frac{d'}{4}$ map. We conclude that $4$ divides $d'$.

A similar argument for $s=2$ and $t \geq 2$ shows that $t$ must divide $d'$.
\end{proof}

\subsection{Proof of Theorem \ref{main-th1}}

\begin{proof}
The necessity follows from Proposition \ref{main-p1}. For sufficiency, suppose there exists a degree $d'$ branched cover $g:S^{2}\rightarrow S^{2}$ with branching data:
\begin{itemize}
    \item Partitions $[a_{11},\ldots,a_{1r_{1}}]$ and $[a_{21},\ldots,a_{2r_{2}}]$
    \item Data obtained by splitting each $A_k$ ($k \geq 3$) into $s$ partitions of $d'$
\end{itemize}

By Lemma \ref{main-le1}, we may assume $g$ is a rational map with:
\begin{itemize}
    \item Branching points at $0$ and $\infty$ corresponding to the first two partitions
    \item Other branching points appropriately assigned
\end{itemize}
Then $g^s$ yields the desired map.
\end{proof}

\subsection{Proof of Theorem \ref{main-th2}}

\begin{proof}
The necessity is established in Proposition \ref{main-p1}. For sufficiency, assume there exists a degree $d'/t$ branched cover $g$ with branching data:
\begin{itemize}
    \item Two partitions obtained from $[a_{31},\ldots,a_{3r_{3}}]$, denoted by $[\widetilde{a}^{1}_{31},\ldots]$,  and $[\widetilde{a}^{2}_{31},\ldots]$
    \item Data from splitting $[a_{11},\ldots,a_{1r_{1}}]$ and $[a_{21},\ldots,a_{2r_{2}}]$ into $t$ partitions of $d'/t$
    \item Data from splitting each $A_k$ ($k \geq 4$) into $2t$ partitions of $d'/t$
\end{itemize}

By Lemma \ref{main-le1}, we may take $g$ to be a rational map of degree $d'/t$ with branch points at $0$ and $\infty$ corresponding to the two partitions $[\widetilde{a}^{1}_{31},\ldots]$,  and $[\widetilde{a}^{2}_{31},\ldots]$,  while the other branch points are assigned suitable complex numbers. Then,
$g^t$ produces a map  of degree $d'$ with partitions $[a_{11},\ldots,a_{1r_{1}}]$, $[a_{21},\ldots,a_{2r_{2}}]$, $[t\widetilde{a}^{1}_{31},\ldots]$, $[t\widetilde{a}^{2}_{31},\ldots]$, and splitting data from $A_{k}(k \geq 4)$.\par

By Lemma \ref{main-le1}, there exists a rational map $h$ with branching points at $\pm 1$ (for the partitions $[t\widetilde{a}^{1}_{31},\ldots]$ and $[t\widetilde{a}^{2}_{31},\ldots]$) and $0,\infty$ (for the the partitions $[a_{11},\ldots,a_{1r_{1}}]$ and  $[a_{21},\ldots,a_{2r_{2}}]$), while the other branch points are assigned suitable complex numbers. Then $h^2$ yields the desired map.

\end{proof}

\subsection{Proof of Theorem \ref{main-th3}}

\begin{proof}
The necessity follows from Proposition \ref{main-p1}. The sufficiency is similar as the proof of Theorem \ref{main-th2}.
\end{proof}

\section{Some Examples}
As an application of our main results, we present several new families of exceptional branching data for maps with at least three branch points. These examples demonstrate the power of Corollary~\ref{Co-1} in identifying exceptional cases. In fact, one can give more exceptional branching data using our results in this paper.

\begin{example}
Let $d=sk$ be a positive integer with $s\geq2$ and $k\geq2$, and let $[a_{1},\ldots,a_{r_{1}}]$ and $[a_{r_{1}+1},\ldots,a_{2k(s-1)+2}]$ be two nontrivial partitions of $d$. If $a_{i}\geq k+1$ for some $i$, then the candidate branching data of
$$(d,\{[a_{1},\ldots,a_{r_{1}}],[a_{r_{1}+1},\ldots,a_{2k(s-1)+2}],
[\underbrace{s,\ldots,s}_k],[\underbrace{s,\ldots,s}_k]\})$$
 is exceptional.
\begin{proof}
Since $a_{i}\geq k+1>k$, according to \textbf{Corollary \ref{Co-1}}, the data is  exceptional.
\end{proof}
\end{example}

\begin{example}
Let $d=sk$ be a positive integer with $s\geq2$ and $k\geq2$,  and let $[a_{1},\ldots,a_{r_{1}}]$, $[a_{r_{1}+1},\ldots,a_{r_{2}}]$ and $[a_{r_{2}+1},\ldots,a_{(3s-2)k+2}]$ be three nontrivial partitions of $d$. If $a_{i}\geq k+1$ for some $i$, then the candidate branching data of
$$(d,\{[a_{1},\ldots,a_{r_{1}}],[a_{r_{1}+1},\ldots,a_{r_{2}}],
[a_{r_{2}+1},\ldots,a_{(3s-2)k+2}],[\underbrace{s,\ldots,s}_k],
[\underbrace{s,\ldots,s}_k]\})$$
is exceptional.
\begin{proof}
Since $a_{i}\geq k+1>k$, according to \textbf{Corollary \ref{Co-1}}, the data is  exceptional.
\end{proof}
\end{example}

\begin{example}
Let $d=sk$ be a positive integer with $s\geq2$ and $k\geq2$, and let $[a_{1},\ldots,a_{r_{1}}]$, $[a_{r_{1}+1},\ldots,a_{r_{2}}]$, $[a_{r_{2}+1},\ldots,a_{r_{3}}]$,  and $[a_{r_{3}+1},\ldots,a_{(4s-2)k+2}]$ be four nontrivial partitions of $d$. If $a_{i}\geq k+1$ for some $i$, then the candidate branching data of
$$(d,\{[a_{1},\ldots,a_{r_{1}}],\ldots,
[a_{r_{3}+1},\ldots,a_{(4s-2)k+2}],[\underbrace{s,\ldots,s}_k],[\underbrace{s,\ldots,s}_k]\})$$
is exceptional.
\begin{proof}
Since $a_{i}\geq k+1>k$, according to \textbf{Corollary \ref{Co-1}}, the data is  exceptional.
\end{proof}
\end{example}

\begin{example}
Let $d=sk$ be a positive integer with $s\geq2$ and $k\geq2$, and let $[a_{1},\ldots,a_{r_{1}}]$, $[a_{r_{1}+1},\ldots,a_{r_{2}}]$, \ldots, $[a_{r_{3}+1},\ldots,a_{r_{4}}]$ and $[a_{r_{4}+1},\ldots,a_{(5s-2)k+2}]$ be five nontrivial partitions of $d$. If $a_{i}\geq k+1$ for some $i$, then the candidate branching data of
$$(d,\{[a_{1},\ldots,a_{r_{1}}],\ldots,
[a_{r_{4}+1},\ldots,a_{(5s-2)k+2}],[\underbrace{s,\ldots,s}_k],[\underbrace{s,\ldots,s}_k]\})$$
is exceptional.
\begin{proof}
Since $a_{i}\geq k+1>k$, according to \textbf{Corollary \ref{Co-1}}, the data is  exceptional.
\end{proof}
\end{example}

\begin{example}
Let $d=sk$ be a positive integer with $s\geq2$ and $k\geq2$, and let $[a_{1},\ldots,a_{r_{1}}]$, $[a_{r_{1}+1},\ldots,a_{r_{2}}]$, \ldots, $[a_{r_{4}+1},\ldots,a_{r_{5}}]$ and $[a_{r_{5}+1},\ldots,a_{(6s-2)k+2}]$ be six nontrivial partitions of $d$. If $a_{i}\geq k+1$ for some $i$, then the candidate branching data of
$$(d,\{[a_{1},\ldots,a_{r_{1}}],\ldots,
[a_{r_{5}+1},\ldots,a_{(6s-2)k+2}],[\underbrace{s,\ldots,s}_k],[\underbrace{s,\ldots,s}_k]\})$$
is exceptional.
\begin{proof}
Since $a_{i}\geq k+1>k$, according to \textbf{Corollary \ref{Co-1}}, the data is  exceptional.
\end{proof}
\end{example}

In general, we have the following example.
\begin{example}
Let $d=sk$ be a positive integer with $s\geq2$ and $k\geq2$, and let $[a_{1},\ldots,a_{r_{1}}]$, \ldots, $[a_{r_{t-2}+1},\ldots,a_{r_{t-1}}]$ and $[a_{r_{t-1}+1},\ldots,a_{(ts-2)k+2}]$ be $t$ nontrivial partitions of $d$. If $a_{i}\geq k+1$ for some $i$, then the candidate branching data of
$$(d,\{[a_{1},\ldots,a_{r_{1}}],\ldots,
[a_{r_{t-1}+1},\ldots,a_{(ts-2)k+2}],[\underbrace{s,\ldots,s}_k],[\underbrace{s,\ldots,s}_k]\})$$
is exceptional.
\begin{proof}
Since $a_{i}\geq k+1>k$, according to \textbf{Corollary \ref{Co-1}}, the data is  exceptional.
\end{proof}
\end{example}


\textbf{Declarations}

\textbf{Data Availability Statement}  This manuscript has no associated data.

\textbf{Competing interests} On behalf of all authors, the corresponding author declares that there is no conflict of interest.


\noindent
Yingjie Meng\\
School of Mathematical Sciences, Henan University,  Kaifeng 475004 P.R. China\\
Email: mengyingjie@henu.edu.cn\\
Zhiqiang Wei\\
School of Mathematics and Statistics, Henan University, Kaifeng 475004 P.R. China\\
Center for Applied Mathematics of Henan Province, Henan University, Zhengzhou 450046 P.R. China\\
Email: weizhiqiang15@mails.ucas.edu.cn ~or~10100123@vip.henu.edu.cn\\
Chuankai Zhou\\
School of Mathematical Sciences, Henan University,  Kaifeng 475004 P.R. China\\
Email: zhouchuankai@henu.edu.cn\\
\end{document}